\documentclass[twoside,american,DIV12,listtotoc,bibtotoc,idxtotoc]{scrartcl}
\usepackage[T1]{fontenc}
\usepackage[latin1]{inputenc}
\usepackage{babel}
\usepackage{verbatim}
\usepackage{prettyref}
\usepackage{amsthm}
\usepackage{amsmath}
\usepackage{amssymb}
\usepackage[unicode=true,pdfusetitle,
 bookmarks=true,bookmarksnumbered=false,bookmarksopen=false,
 breaklinks=false,pdfborder={0 0 1},backref=false,colorlinks=false]
 {hyperref}

\makeatletter
\theoremstyle{plain}
\newtheorem{thm}{\protect\theoremname}[section]
  \theoremstyle{definition}
  \newtheorem{defn}[thm]{\protect\definitionname}
  \theoremstyle{remark}
  \newtheorem{rem}[thm]{\protect\remarkname}
  \theoremstyle{definition}
  \newtheorem{example}[thm]{\protect\examplename}
  \theoremstyle{plain}
  \newtheorem{lem}[thm]{\protect\lemmaname}

\usepackage{helvet}
\usepackage[T1]{fontenc}

\usepackage{amsmath}
\usepackage{setspace}
\usepackage{amssymb}
\usepackage{enumerate}
\usepackage{url}

\makeatletter

\usepackage[T1]{fontenc}    

\usepackage{a4wide}        
\usepackage{tu-preprint}
\usepackage{amsmath}

\usepackage{euscript}
\usepackage{mathtools}
\allowdisplaybreaks[1] 
\usepackage{amsfonts}
\newenvironment{keywords}{ \noindent\footnotesize\textbf{Keywords and phrases:}}{}

\newenvironment{class}{\noindent\footnotesize\textbf{Mathematics subject classification 2010:}}{}

\usepackage{color,tu-preprint}

\newcommand*{\trace}{\operatorname{trace}}

\newcommand*{\dive}{\operatorname{div}}

\newcommand*{\Grad}{\operatorname{Grad}}
\newcommand*{\Dive}{\operatorname{Div}}

\newcommand*{\grad}{\operatorname{grad}}

\renewcommand*{\i}{\mathrm{i}}

\DeclareMathAccent{\Circ}{\mathalpha}{operators}{"17}
\newcommand{\interior}[1]{\Circ{#1}}

\renewcommand{\Re}{\operatorname{\mathfrak{Re}}}

\renewcommand{\hat}{\widehat}
\renewcommand{\tilde}{\widetilde}
\renewcommand*{\epsilon}{\varepsilon}
\renewcommand*{\rho}{\varrho}

\makeatother

\usepackage{babel}
\institut{Institut f\"ur Analysis}

\preprintnumber{MATH-AN-12-2012}

\preprinttitle{On a Class of Boundary Control Problems.}

\author{Rainer Picard, Sascha Trostorff \& Marcus Waurick}

\AtBeginDocument{
  
}

\makeatother

  \providecommand{\definitionname}{Definition}
  \providecommand{\examplename}{Example}
  \providecommand{\lemmaname}{Lemma}
  \providecommand{\remarkname}{Remark}
\providecommand{\theoremname}{Theorem}

\begin{document}
\makepreprinttitlepage

\author{ Rainer Picard, \\ Sascha Trostorff, \\ Marcus Waurick \\ Institut f\"ur Analysis, Fachrichtung Mathematik\\ Technische Universit\"at Dresden\\ Germany\\ rainer.picard@tu-dresden.de\\ sascha.trostorff@tu-dresden.de\\ marcus.waurick@tu-dresden.de }

\title{On a Class of Boundary Control Problems.}

\maketitle
\begin{abstract} \textbf{Abstract.} We discuss a class of linear
control problems in a Hilbert space setting, which covers diverse
systems such as hyperbolic and parabolic equations with boundary control
and boundary observation even including memory terms. We introduce
abstract boundary data spaces in which the control and observation
equations can be formulated without strong geometric constraints on
the underlying domain. The results are applied to a boundary control
problem for the equations of visco-elasticity. \end{abstract}

\begin{keywords} linear control systems, well-posedness, evolutionary
equations, memory \end{keywords}

\begin{class} 93C05 (Linear systems), 93C20 (Systems governed by
partial differential equations), 93C25 (Systems in abstract spaces)
 \end{class}

\newpage

\tableofcontents{} 

\newpage

\section{Introduction}

Linear Control systems are typically given by a differential equation,
linking the state $x$ and the control $u$ 
\[
\dot{x}(t)=Ax(t)+Bu(t),\quad t\in\mathbb{R}_{>0},
\]

usually completed by an initial condition $x(0+)=x_{0},$ and an algebraic
equation linking state, control and the observation $y$

\[
y(t)=Cx(t)+Du(t)\quad t\in\mathbb{R}_{>0},
\]

where $A,B,C$ and $D$ are matrices of appropriate sizes. Rewriting
these equations as one system acting on the whole real line $\mathbb{R}$
instead of the positive half-line $\mathbb{R}_{>0}$ we end up with
an differential-algebraic system of the form 
\[
\partial_{0}\left(\begin{array}{cc}
1 & 0\\
0 & 0
\end{array}\right)\left(\begin{array}{c}
x\\
y
\end{array}\right)+\left(\begin{array}{cc}
0 & 0\\
-C & 1
\end{array}\right)\left(\begin{array}{c}
x\\
y
\end{array}\right)+\left(\begin{array}{cc}
-A & 0\\
0 & 0
\end{array}\right)\left(\begin{array}{c}
x\\
y
\end{array}\right)=\delta\otimes\left(\begin{array}{c}
x_{0}\\
0
\end{array}\right)+\left(\begin{array}{c}
B\\
D
\end{array}\right)u,
\]

where $\partial_{0}$ denotes the derivative with respect to time
and the initial condition for the state variable $x$ transforms into
an additional Dirac-$\delta$-source term on the right hand side.
Systems of this form have been studied in the finite- and infinite-dimensional
case in various works. In the infinite-dimensional case the operators
$B,C$ and $D$, acting on some suitable Banach- or Hilbert-spaces,
are usually assumed to be bounded, while the operator $A$ is a generator
of a $C_{0}$-semigroup. In this case the solution theory is rather
straightforward. However, in case of boundary control and observation
it turns out that the operators $B$ and $C$ are in general unbounded
and hence, a more sophisticated theory is needed. The classical approach
is to consider so-called admissible operators $B$ and $C$ as it
was done for instance in \cite{Salamon1987,Salamon1989,Curtain1989,Engel1998,LasTrigg2000-1,LasTrigg2000-2,Weiss1989,Weiss1989_representation,Tucsnak_Weiss2003,Jacob2004}.
We focus on a class of linear control problems, where the operators
$B$ and $C$ are unbounded, but in our approach the admissibility
for these operators has not to be verified.\\
The solution theory provided in this article is based on the theory
of evolutionary equations as they were considered in \cite{Picard}.
As it was shown in \cite{Picard2012_comprehensive_control,Picard2012_conservative},
linear control systems (including the case of unbounded operators
$B$ and $C$) are just a subclass of evolutionary equations. In this
note we will generalize the solution theory presented in \cite{Picard2012_comprehensive_control}
to a broader class of so-called \emph{(linear) material laws}. This
generalization allows us to study control problems including delay
terms. As in \cite{Picard2012_comprehensive_control} we introduce
abstract boundary data spaces, which enable us to formulate boundary
control and observation equations without strong smoothness assumptions
on the boundary of the underlying domain. Indeed, it will suffice
to guarantee a Poincare-type estimate for the involved differential
operators.\\
Section \ref{sec:Functional-analytic-framework} recalls some preliminaries
on evolutionary equations, linear material laws and extrapolation
spaces (so-called Sobolev chains) and we refer to \cite{Picard,Picard_McGhee,Kalauch2011}
for the proofs and a deeper study of the related topics. \\
In Section \ref{sec:Control-Systems-as} we introduce the notion of
linear control systems, which will be a special case of the broader
class of abstract evolutionary equations. In contrast to \cite{Picard2012_comprehensive_control}
we will generalize the class of possible control problems to the case
of arbitrary material laws (while in \cite{Picard2012_comprehensive_control}
just the so-called (P)-degenerate case, cf. \cite{Picard}, was treated).
We provide a well-posedness result for this class, which is in essence
just an application of \cite[Solution theory]{Picard}, and show the
causality of the solution operator. \\
Boundary control problems are introduced in Section \ref{sec:Boundary-control}
and we show that they fit into the abstract class of linear control
problems introduced previously. In order to formulate boundary control
and observation equations, without imposing strong smoothness constraints
on the domain, we introduce abstract traces and recall the notion
of abstract boundary data spaces. Finally, we apply our findings to
a boundary control problem for the equations of visco-elasticity.

\section{Preliminaries\label{sec:Functional-analytic-framework}}

In this section we recall the notion of evolutionary equations. Following
\cite{Kalauch2011} we begin to introduce the time derivative $\partial_{0}$
as a boundedly invertible operator on an exponentially weighted $L_{2}$-space. 
\begin{defn}
For $\nu\in\mathbb{R}_{>0}$ we denote by $H_{\nu,0}(\mathbb{R})$
the space of all square-integrable functions%
\footnote{Throughout we identify the equivalence classes induced by the equality
almost everywhere with their representatives.%
} with respect to the exponentially weighted Lebesgue-measure $\exp(-2\nu t)\mbox{ d}t$,
endowed with the inner product given by 
\[
\langle f|g\rangle_{\nu,0}\coloneqq\intop_{\mathbb{\mathbb{R}}}f(t)^{\ast}g(t)\exp(-2\nu t)\mbox{ d}t\quad(f,g\in H_{\nu,0}(\mathbb{R})).
\]
We define the operator $\partial_{0,\nu}$ on $H_{\nu,0}(\mathbb{R})$
as the closure of the derivative 
\begin{align*}
\partial_{0,\nu}|_{C_{c}^{\infty}(\mathbb{R})}\colon C_{c}^{\infty}(\mathbb{R})\subseteq H_{\nu,0}(\mathbb{R}) & \to H_{\nu,0}(\mathbb{R})\\
\phi & \mapsto\phi'.
\end{align*}
This operator is normal with $\Re\partial_{0,\nu}=\nu$ and hence,
$0\in\rho(\partial_{0,\nu})$ with $\|\partial_{0,\nu}^{-1}\|\leq\frac{1}{\nu}$
. If the choice of $\nu$ is clear from the context we will omit the
additional index $\nu.$
\end{defn}
Clearly, we can extend the operator $\partial_{0}$ to the space of
$H$-valued functions $H_{\nu,0}(\mathbb{R};H)$, where $H$ is a
Hilbert space, in a canonical way. 
\begin{rem}
$\,$

\begin{enumerate}[(a)]

\item For $u\in H_{\nu,0}(\mathbb{R})$ the function $\partial_{0}^{-1}u$
is given by 
\[
\left(\partial_{0}^{-1}u\right)(x)=\intop_{-\infty}^{x}u(t)\,\mathrm{d}t\quad(x\in\mathbb{R}\mbox{ a.e.}).
\]
This especially yields the causality of $\partial_{0}^{-1}$.%
\footnote{A mapping $F:H_{\nu,0}(\mathbb{R};H)\to H_{\nu,0}(\mathbb{R};H)$,
where $H$ is an arbitrary Hilbert space, is called \emph{causal},
if\emph{ }for each $a\in\mathbb{R}$ it holds $\chi_{(-\infty,a]}(m)F\chi_{(-\infty,a]}(m)=\chi_{(-\infty,a]}(m)F,$
where by $\chi_{(-\infty,a]}(m)$ we denote the operator on $H_{\nu,0}(\mathbb{R};H)$
mapping a function $f$ to the truncated function $t\mapsto\chi_{(-\infty,a]}(t)f(t).$ %
} 

\item Let $\mathcal{F}:L_{2}(\mathbb{R})\to L_{2}(\mathbb{R})$ denote
the Fourier-transform and define for $\nu>0$ the operator $e^{-\nu m}:H_{\nu,0}(\mathbb{R})\to L_{2}(\mathbb{R})$
by $\left(e^{-\nu m}f\right)(x)=e^{-\nu x}f(x)$ for almost every
$x\in\mathbb{R},$ which is obviously unitary. Then we define the
\emph{Fourier-Laplace-transform }$\mathcal{L}_{\nu}\coloneqq\mathcal{F}e^{-\nu m}:H_{\nu,0}(\mathbb{R})\to L_{2}(\mathbb{R})$,
which gives the following spectral representation of the derivative
$\partial_{0,\nu}$:
\[
\partial_{0,\nu}=\mathcal{L}_{\nu}^{\ast}(\i m+\nu),
\]
where $m$ denotes the ``multiplication-by-the-argument'' operator
on $L_{2}(\mathbb{R})$ with maximal domain.

\end{enumerate}
\end{rem}
With the spectral representation of $\partial_{0}$ we can define
so-called linear material laws (cf. \cite{Picard}).
\begin{defn}
Let $r>0,\, H$ an arbitrary Hilbert space and $M:B_{\mathbb{C}}(r,r)\to L(H)$
be bounded and analytic. Then we define the operator $M\left(\frac{1}{\i m+\nu}\right)$
on $L_{2}(\mathbb{R})$ for $\nu>\frac{1}{2r}$ by 
\[
\left(M\left(\frac{1}{\i m+\nu}\right)f\right)(x)=M\left(\frac{1}{\i x+\nu}\right)f(x)\quad(x\in\mathbb{R}\mbox{ a.e.})
\]
and the \emph{linear material law} $M(\partial_{0}^{-1})\in L(H_{\nu,0}(\mathbb{R};H))$
for $\nu>\frac{1}{2r}$ by 
\[
M\left(\partial_{0}^{-1}\right)\coloneqq\mathcal{L}_{\nu}^{\ast}M\left(\frac{1}{\i m+\nu}\right)\mathcal{L}_{\nu}.
\]
\end{defn}
\begin{rem}
Due to the analyticity of $M$ we obtain by a Paley-Wiener result
that the operator $M(\partial_{0}^{-1})$ is causal.
\end{rem}
Note that any densely defined closed linear operator $A$ defined
in a Hilbert space $H$ gives rise to densely defined closed linear
operator in $H_{\nu,0}(\mathbb{R};H)$ defined as the canonical extension
of the operator acting as $(Af)(t)\coloneqq Af(t)$ for all $t\in\mathbb{R}$
and simple functions $f$ taking values in the domain of $A$. Henceforth,
we will identify $A$ with its extension without further notice.
\begin{thm}[{Solution theory for evolutionary equations \cite[Solution theory]{Picard}}]
\label{thm:sol_theory} Let $H$ be a Hilbert space and $A:D(A)\subseteq H\to H$
be skew-selfadjoint. Furthermore let $r>0$ and $M:B_{\mathbb{C}}(r,r)\to L(H)$
be analytic, bounded and assume that there exists $c>0$ such that
for all $z\in B_{\mathbb{C}}(r,r)$ 
\begin{equation}
\Re z^{-1}M(z)\geq c.\label{eq:pos_def}
\end{equation}
Then there exists $\nu_{0}>0$ such that for all $\nu\geq\nu_{0}$
the \emph{evolutionary equation 
\begin{equation}
\left(\overline{\partial_{0}M(\partial_{0}^{-1})+A}\right)u=f,\label{eq:evol_eq}
\end{equation}
}admits for every $f\in H_{\nu,0}(\mathbb{R};H)$ a unique solution
$u\in H_{\nu,0}(\mathbb{R};H),$ which depends continuously on $f.$
More precisely, $0\in\rho\left(\overline{\partial_{0}M(\partial_{0}^{-1})+A}\right)$
and the \emph{solution operator $\left(\overline{\partial_{0}M(\partial_{0}^{-1})+A}\right)^{-1}$}
is causal.
\end{thm}
Next we introduce the concept of Sobolev-chains. For the proofs and
further details we refer to \cite[Chapter 2]{Picard_McGhee}.
\begin{defn}
Let $H$ be a Hilbert space and $C:D(C)\subseteq H\to H$ be a densely
defined, closed linear operator with $0\in\rho(C).$ For $k\in\mathbb{Z}$
we define $H_{k}(C)$ as the completion of the domain $D(C^{k})$
with respect to the norm $|C^{k}\cdot|_{H}.$ Then $\left(H_{k}(C)\right)_{k\in\mathbb{Z}}$
is a sequence of Hilbert spaces with $H_{k}(C)\hookrightarrow H_{k-1}(C)$
for $k\in\mathbb{Z}.$ The sequence $(H_{k}(C))_{k\in\mathbb{Z}}$
is called the \emph{Sobolev-chain of $C$. }For each $k\in\mathbb{Z}$
the operator $C:H_{|k|+1}(C)\subseteq H_{k+1}(C)\to H_{k}(C)$ possesses
a unitary extension to $H_{k+1}(C)$, which will be again denoted
by $C$. Furthermore $H_{k}(C)^{\ast}$ can be identified with $H_{-k}(C^{\ast})$
for each $k\in\mathbb{Z}$ via a unitary operator.\end{defn}
\begin{rem}
$\,$

\begin{enumerate}[(a)]

\item Let $H_{0},H_{1}$ be two Hilbert spaces over the same field
and $A:D(A)\subseteq H_{0}\to H_{1}$ be densely defined, closed and
linear. For each $k\in\mathbb{Z}$ the operator
\[
A:H_{|k|+1}(|A|+\i)\subseteq H_{k+1}(|A|+\i)\to H_{k}(|A^{\ast}|+\i)
\]
has a unique continuous extension to $H_{k+1}(|A|+\i).$

\item Let $\left(H_{k}(C)\right)_{k\in\mathbb{Z}}$ be a Sobolev-chain
associated to some operator $C$ and $A:H_{1}(C)\to H$ be linear
and bounded, where $H$ denotes an arbitrary Hilbert-space. Then the
operator $A^{\diamond}:H\to H_{-1}(C^{\ast})$ is defined as the dual
operator of $A$, where we identify the dual of $H$ with $H$ and
$H_{1}(C)^{\ast}$ is identified with $H_{-1}(C^{\ast}).$

\end{enumerate}
\end{rem}

\section{Abstract linear Control Systems\label{sec:Control-Systems-as} }

In this section we introduce the shape of linear control systems and
show that they fit into the class of evolutionary equations introduced
in the previous section. We consider a densely defined closed linear
operator $F:D(F)\subseteq H_{0}\to H_{1}$ for two Hilbert spaces
$H_{0}$ and $H_{1}$. Furthermore let $U$ and $Y$ be Hilbert spaces,
which will serve as control and observation space, respectively.
\begin{defn}
Let $M_{1,i2}\in L(Y;H_{i})$ and $M_{1,2i}\in L(H_{i};Y)$ for $i\in\{0,1\}$
and $M_{1,22}\in L(Y;Y)$. Let 
\begin{equation}
M(z)\coloneqq\left(\begin{array}{cc}
K(z) & \left(\begin{array}{c}
0\\
0
\end{array}\right)\\
\left(\begin{array}{cc}
0 & 0\end{array}\right) & 0
\end{array}\right)+z\left(\begin{array}{cc}
0 & \left(\begin{array}{c}
M_{1,02}\\
M_{1,12}
\end{array}\right)\\
\left(\begin{array}{cc}
M_{1,20} & M_{1,21}\end{array}\right) & M_{1,22}
\end{array}\right)\quad(z\in B_{\mathbb{C}}(r,r)),\label{eq:material_law}
\end{equation}
where $K:B_{\mathbb{C}}(r,r)\to L(H_{0}\oplus H_{1})$ is a linear
material law. An\emph{ abstract} \emph{linear control system }$\mathcal{C}_{M,F,B}$
is an evolutionary equation of the form\emph{ 
\[
\left(\partial_{0}M(\partial_{0}^{-1})+\left(\begin{array}{ccc}
0 & F^{\ast} & 0\\
-F & 0 & 0\\
0 & 0 & 0
\end{array}\right)\right)\left(\begin{array}{c}
x\\
\xi\\
y
\end{array}\right)=f+Bu.
\]
}Here $f\in H_{\nu,-\infty}(\mathbb{R};H_{0}\oplus H_{1}\oplus Y)$
is an arbitrary source term, $u\in H_{\nu,0}(\mathbb{R};U)$ is the
control and $B\in L(H_{\nu,0}(\mathbb{R};U);H_{\nu,0}(\mathbb{R};(H_{0}\oplus H_{1}\oplus Y)))$
is the control operator. We call an abstract linear control system
\emph{well-posed,} if the operator 
\[
\partial_{0}M(\partial_{0}^{-1})+\left(\begin{array}{ccc}
0 & F^{\ast} & 0\\
-F & 0 & 0\\
0 & 0 & 0
\end{array}\right)\subseteq H_{\nu,0}(\mathbb{R};H_{0}\oplus H_{1}\oplus Y)\oplus H_{\nu,0}(\mathbb{R};H_{0}\oplus H_{1}\oplus Y)
\]
possesses a densely defined bounded inverse for sufficiently large
$\nu.$ The continuation to the space $H_{\nu,0}(\mathbb{R};H_{0}\oplus H_{1}\oplus Y)$
of the inverse is called \emph{solution operator}.
\end{defn}
It is clear that an abstract linear control system $\mathcal{C}_{M,F,B}$
is of the form \prettyref{eq:evol_eq} given in Theorem \ref{thm:sol_theory}
with $A=\left(\begin{array}{ccc}
0 & F^{\ast} & 0\\
-F & 0 & 0\\
0 & 0 & 0
\end{array}\right)$. Hence, the solution theory for evolutionary equations is applicable.
It is obvious that if $M$ satisfies the condition (\ref{eq:pos_def}),
then so does $N$ and the operator $\Re M_{1,22}$ is strictly positive
definite. However, the latter is not a sufficient condition for the
positive definiteness of $\Re z^{-1}M(z).$ We define the operator
$J\in L(Y;H_{0}\oplus H_{1})$ by 
\[
J\coloneqq\frac{1}{2}\left(\begin{array}{c}
M_{1,02}+M_{1,20}^{\ast}\\
M_{1,12}+M_{1,21}^{\ast}
\end{array}\right).
\]

\begin{thm}
\label{thm:well_posedness_control}Let $\mathcal{C}_{M,F,B}$ be an
abstract linear control system. Assume that $\Re z^{-1}N(z)\geq c_{0}>0$
and $\Re M_{1,22}\geq c_{1}>0$. Assume that there is $\delta>0$
such that $c_{0}-\delta\|J\|>0$ and $c_{1}-\frac{1}{\delta}\|J\|>0.$
Then $\mathcal{C}_{M,F,B}$ is well-posed and the solution operator
is causal.\end{thm}
\begin{proof}
Since,
\[
\Re\left(\begin{array}{cc}
0 & \left(\begin{array}{c}
M_{1,02}\\
M_{1,12}
\end{array}\right)\\
\left(\begin{array}{cc}
M_{1,20} & M_{1,21}\end{array}\right) & M_{1,22}
\end{array}\right)=\left(\begin{array}{cc}
0 & J\\
J^{\ast} & \Re M_{1,22}
\end{array}\right)
\]
we get for $w\coloneqq(x,\xi,y)\in H_{0}\oplus H_{1}\oplus Y$
\begin{align*}
\Re\langle z^{-1}M(z)w|w\rangle & =\Re\left\langle z^{-1}N(z)\left(\begin{array}{c}
x\\
\xi
\end{array}\right)\left|\left(\begin{array}{c}
x\\
\xi
\end{array}\right)\right.\right\rangle +\langle\Re M_{1,22}y|y\rangle+2\Re\left\langle Jy\left|\left(\begin{array}{c}
x\\
\xi
\end{array}\right)\right.\right\rangle \\
 & \geq c_{0}\left|\left(\begin{array}{c}
x\\
\xi
\end{array}\right)\right|^{2}+c_{1}|y|^{2}-2\|J\|\left|\left(\begin{array}{c}
x\\
\xi
\end{array}\right)\right||y|\\
 & \geq(c_{0}-\delta\|J\|)\left|\left(\begin{array}{c}
x\\
\xi
\end{array}\right)\right|^{2}+\left(c_{1}-\frac{1}{\delta}\|J\|\right)|y|^{2}.
\end{align*}
The assertion then follows by Theorem \ref{thm:sol_theory}.
\end{proof}

\section{Boundary Control Systems\label{sec:Boundary-control}}

This section is devoted to the study of boundary control systems.
At first we show how boundary control and observation equations can
be handled within the framework presented in the previous sections.
This will mainly be done by a particular choice for the unbounded
operator $F.$ As it was pointed out in \cite[Subsection 5.1]{Picard2012_comprehensive_control},
the resulting class of control systems can be interpreted as a generalization
of a subclass of port-Hamiltonian systems (cf. \cite{Jacob_Zwart2012})
to the higher dimensional case. Moreover, we recall the notion of
so-called boundary data spaces, introduced in \cite[Subsection 5.2]{Picard2012_comprehensive_control},
as well as abstract traces, which enable us to treat boundary values
as suitable distributions belonging to some extrapolation space. \\
First let us fix some notation. For Hilbert spaces $H_{0},\ldots,H_{n}$
we define for $i\in\{0,\ldots,n\}$ the operator 
\[
\pi_{H_{i}}:H_{0}\oplus\ldots\oplus H_{n}\to H_{i}
\]
as the orthogonal projection on $H_{i}.$ Note that then $\pi_{H_{i}}^{\ast}$
is the canonical embedding from $H_{i}$ to $H_{0}\oplus\ldots\oplus H_{n}.$
\\

We begin this section with an illustrative example.
\begin{example}
Let $\Omega\subseteq\mathbb{R}^{n}$ be an arbitrary domain and define
the operators $\interior\grad$ and $\interior\dive$ as the closures
of 
\begin{align*}
\grad|_{C_{c}^{\infty}(\Omega)}:C_{c}^{\infty}(\Omega)\subseteq L_{2}(\Omega) & \to L_{2}(\Omega)^{n}\\
\phi & \mapsto\left(\partial_{i}\phi\right)_{i\in\{1,\ldots,n\}}
\end{align*}
and 
\begin{align*}
\dive|_{C_{c}^{\infty}(\Omega)^{n}}:C_{c}^{\infty}(\Omega)^{n}\subseteq L_{2}(\Omega)^{n} & \to L_{2}(\Omega)\\
\left(\phi_{i}\right)_{i\in\{1,\ldots,n\}} & \mapsto\sum_{i=1}^{n}\partial_{i}\phi,
\end{align*}
respectively. These operators are formally skew-adjoint, i.e., $\interior\grad\subseteq-\left(\interior\dive\right)^{\ast}\eqqcolon\grad$
and $\interior\dive\subseteq-\left(\interior\grad\right)^{\ast}\eqqcolon\dive.$
Then, using the extrapolation spaces of the operators $|\dive|+\i$
and $|\grad|+\i$ we define the \emph{Dirichlet-trace }and the \emph{Neumann-trace}
by 
\begin{align*}
\gamma_{\grad}:H_{1}(|\grad|+\i) & \to H_{-1}(|\dive|+\i)\\
u & \mapsto\left(\grad-\interior\grad\right)u
\end{align*}
and 
\begin{align*}
\gamma_{\dive}:H_{1}(|\dive|+\i) & \to H_{-1}(|\grad|+\i)\\
\zeta & \mapsto\left(\dive-\interior\dive\right)\zeta,
\end{align*}
respectively. Note that in the case of a smooth boundary, the distributions
$\gamma_{\grad}u$ and $\gamma_{\dive}\zeta$ for $u\in H_{1}(|\grad|+\i)$
and $\zeta\in H_{1}(|\dive|+\i)$ are supported on $\partial\Omega.$
More precisely with the help of the divergence theorem, 
\[
\langle\gamma_{\grad}u|\zeta\rangle=\intop_{\partial\Omega}u^{\ast}\zeta\cdot n\mbox{ d}S=\langle u|\gamma_{\dive}\zeta\rangle,
\]
where $n$ denotes the unit outward normal and $S$ the surface measure
on $\partial\Omega$. Note that $\gamma_{\grad}u=0$ if and only if
$u\in D(\interior\grad)$ and $\gamma_{\dive}\zeta=0$ if and only
if $\zeta\in D(\interior\dive).$
\end{example}
In the rest of this subsection we generalize the concepts illustrated
in the example above. For that purpose let $H_{0}$ and $H_{1}$ be
two complex Hilbert spaces, $\interior G\subseteq H_{0}\oplus H_{1}$
and $\interior D\subseteq H_{1}\oplus H_{0}$ two densely defined
closed linear operators, which are formally skew-adjoint. We define
$G\coloneqq-\left(\interior D\right)^{\ast}$ and $D\coloneqq-\left(\interior G\right)^{\ast}.$
\begin{defn}[Abstract traces]
 We define the \emph{abstract traces }$\gamma_{G}$ and $\gamma_{D}$
by 
\begin{align*}
\gamma_{G}:H_{1}(|G|+\i) & \to H_{-1}(|D|+\i)\\
v & \mapsto\left(G-\interior G\right)v
\end{align*}
and 
\begin{align*}
\gamma_{D}:H_{1}(|D|+\i) & \to H_{-1}(|G|+\i)\\
w & \mapsto\left(D-\interior D\right)w.
\end{align*}
Furthermore we define the \emph{abstract trace spaces }as the image
spaces of the respective trace operators, i.e. 
\begin{align*}
\mathrm{TR}(G) & \coloneqq\gamma_{G}[H_{1}(|G|+\i)],\\
\mathrm{TR}(D) & \coloneqq\gamma_{D}[H_{1}(|D|+\i)].
\end{align*}

\end{defn}
Clearly the kernels of $\gamma_{G}$ and $\gamma_{D}$ are given by
$H_{1}(|\interior G|+\i)$ and $H_{1}(|\interior D|+\i)$ respectively.
This leads to the following definition.
\begin{defn}[Boundary data spaces]
 We define the \emph{boundary data spaces} $\mathrm{BD}(G)$ and
$\mathrm{BD}(D)$ as 
\[
\mathrm{BD}(G)\coloneqq H_{1}(|\interior G|+\i)^{\bot_{H_{1}(|G|+\i)}}
\]
and 
\[
\mathrm{BD}(D)\coloneqq H_{1}(|\interior D|+\i)^{\bot_{H_{1}(|D|+\i)}}.
\]
\end{defn}
\begin{lem}
The boundary data spaces are given by $\mathrm{BD}(G)=[\{0\}](1-DG)$
and $\mathrm{BD}(D)=[\{0\}](1-GD).$\end{lem}
\begin{proof}
See \cite[Lemma 5.1]{Picard2012_comprehensive_control}.\end{proof}
\begin{thm}
The operators 
\[
\gamma_{G}|_{\mathrm{BD}(G)}:\mathrm{BD}(G)\to\mathrm{TR}(G)
\]
and 
\[
\gamma_{D}|_{\mathrm{BD}(D)}:\mathrm{BD}(D)\to\mathrm{TR(D)}
\]
are unitary.\end{thm}
\begin{proof}
Let $u\in\mathrm{BD}(G).$ Then we get for each $v\in H_{1}(|D|+\i)$
that 
\begin{align*}
|\langle\gamma_{G}u|v\rangle| & =|\langle Gu|v\rangle+\langle u|Dv\rangle|\\
 & =|\langle Gu|v\rangle+\langle DGu|Dv\rangle|\\
 & =|\langle Gu|v\rangle_{H_{1}(|D|+\i)}|\\
 & \leq|Gu|_{H_{1}(|D|+\i)}|v|_{H_{1}(|D|+\i)}
\end{align*}
and hence, 
\[
|\gamma_{G}u|_{H_{-1}(|D|+\i)}\leq|Gu|_{H_{1}(|D|+\i)}=\sqrt{|Gu|^{2}+|DGu|^{2}}=\sqrt{|Gu|^{2}+|u|^{2}}=|u|_{H_{1}(|G|+\i)}.
\]
 On the other hand we have 
\[
\langle\gamma_{G}u|Gu\rangle=\langle Gu|Gu\rangle+\langle u|DGu\rangle=\langle Gu|Gu\rangle+\langle u|u\rangle=|u|_{H_{1}(|G|+\i)}^{2},
\]
which gives $|u|_{H_{1}(|G|+\i)}\leq|\gamma_{G}u|_{H_{-1}(|D|+\i)}.$
That $\gamma_{G}|_{\mathrm{BD}(G)}$ is onto, follows by the definition
of $\mathrm{BD}(G)$ and $\mathrm{TR}(G).$ The assertion for $\gamma_{D}|_{\mathrm{BD}(D)}$
follows by interchanging the roles of $D$ and $G$. 
\end{proof}
Since $G[\mathrm{BD}(G)]\subseteq\mathrm{BD}(D)$ and $D[\mathrm{BD}(D)]\subseteq\mathrm{BD}(G)$
we may consider the following restrictions of $G$ and $D$ 

\begin{align*}
\overset{\bullet}{D}:\mathrm{BD}\left(D\right) & \to\mathrm{BD}\left(G\right)\\
\phi & \mapsto D\phi
\end{align*}

and 
\begin{align*}
\overset{\bullet}{G}:\mathrm{BD}\left(G\right) & \to\mathrm{BD}\left(D\right)\\
\phi & \mapsto G\phi.
\end{align*}
The operators $\overset{\bullet}{D}$ and $\overset{\bullet}{G}$
enjoy the following surprising property.
\begin{thm}
\label{thm:unitary_D_G}We have that
\[
\left(\overset{\bullet}{G}\right)^{*}=\overset{\bullet}{D}=\left(\overset{\bullet}{G}\right)^{-1}.
\]
In particular, $\overset{\bullet}{G}$ and $\overset{\bullet}{D}$
are unitary.\end{thm}
\begin{proof}
See \cite[Theorem 5.2.]{Picard2012_comprehensive_control}.\end{proof}
\begin{rem}
The operator $\gamma_{D}|_{\mathrm{BD}(D)}\stackrel{\bullet}{G}\left(\gamma_{G}|_{\mathrm{BD}(G)}\right)^{-1}:\mathrm{TR}(G)\to\mathrm{TR}(D)$
is unitary and can be interpreted as an abstract version of the Dirichlet-to-Neumann
operator. 
\end{rem}

After these preparations we show, how systems with boundary control
and boundary observation can be treated within the framework of Section
3. For doing so, let $C\in L(H_{1}(|G|+\i);V)$ for some Hilbert space
$V$ and assume that $F$ is given by

\begin{equation}
F\coloneqq\left(\begin{array}{c}
-G\\
C
\end{array}\right):H_{1}(|G|+\i)\subseteq H_{0}(|G|+\i)\to H_{0}(|\interior D|+\i)\oplus V.\label{eq:op_F}
\end{equation}

As in the definition of an abstract linear control system $\mathcal{C}_{M,F,B}$
the adjoint of $F$ comes into play. We compute it explicitly in the
next theorem for the case, when $G$ is assumed to be boundedly invertible.
In applications this requirement can be guaranteed by assuming certain
geometric properties of the underlying domain (e.g. segment property,
Lipschitz boundary and so on).
\begin{thm}
\label{thm:adjoint_F}Let $F$ be given as above and let $G$ be boundedly
invertible. Then 
\begin{align*}
F^{\ast}:D(F^{\ast})\subseteq H_{0}(|\interior D|+\i)\oplus V & \to H_{0}(|G|+\i)\\
(\zeta,w) & \mapsto\interior D\zeta+C^{\diamond}w,
\end{align*}
where $C^{\diamond}$ is the dual operator of $C$ with respect to
the Gelfand-triplet $H_{1}(|G|+\i)\subseteq H_{0}(|G|+\i)\subseteq H_{-1}(|G|+\i)$
and 
\[
D(F^{\ast})=\{(\zeta,w)\in H_{0}(|\interior D|+\i)\oplus V\,|\,\interior D\zeta+C^{\diamond}w\in H_{0}(|G|+\i)\}.
\]
\end{thm}
\begin{proof}
See \cite[Theorem 5.4.]{Picard2012_comprehensive_control}.\end{proof}
\begin{rem}
Let $\mathcal{C}_{M,F,B}$ be an abstract linear control system, where
$F$ is given by (\ref{eq:op_F}). We assume that $G$ is boundedly
invertible. Note that, as a consequence, $\interior D$ is boundedly
invertible as well. An element $(x,(\zeta,w))\in H_{0}(|G|)\oplus\left(H_{0}(|\interior D|)\oplus V\right)$
belongs to the domain of $\left(\begin{array}{cc}
0 & -F^{\ast}\\
F & 0
\end{array}\right)$ if and only if $x\in H_{1}(|G|)$ and $\interior D\zeta+C^{\diamond}w\in H_{0}(|\interior D|).$
The latter is equivalent to 
\begin{equation}
\gamma_{D}(\zeta+(\interior D)^{-1}C^{\diamond}w)=0.\label{eq:bd_cd}
\end{equation}
Recall from \prettyref{eq:material_law} that $M$ is of the form
\[
M(z)=\left(\begin{array}{cc}
K(z) & \left(\begin{array}{c}
0\\
0
\end{array}\right)\\
\left(\begin{array}{cc}
0 & 0\end{array}\right) & 0
\end{array}\right)+z\left(\begin{array}{cc}
0 & \left(\begin{array}{c}
M_{1,02}\\
M_{1,12}
\end{array}\right)\\
\left(\begin{array}{cc}
M_{1,20} & M_{1,21}\end{array}\right) & M_{1,22}
\end{array}\right)
\]
for suitable operators $M_{1,ij}$ and $K:B_{\mathbb{C}}(r,r)\to L(H_{0}(|G|)\oplus H_{0}(|\interior D|)\oplus V).$
Note that due to the block structure of $F^{\ast}$, this operator
has indeed four lines and columns. With $(x,\xi)=(x,(\zeta,w))$,
the third and fourth line of the equation given by $\mathcal{C}_{M,F,B}$
read as 
\[
\partial_{0}\pi_{V}K(\partial_{0}^{-1})\left(\pi_{H_{0}(|G|)}^{\ast}x+\pi_{H_{0}(|\interior D|)}^{\ast}\zeta\right)+\partial_{0}\pi_{V}K(\partial_{0}^{-1})\pi_{V}^{\ast}w+\pi_{V}M_{1,12}y+Cx=\pi_{V}f+\pi_{V}Bu
\]
and 
\[
M_{1,20}x+M_{1,21}\pi_{H_{0}(|\interior D|)}^{\ast}\zeta+M_{1,21}\pi_{V}^{\ast}w+M_{1,22}y=\pi_{Y}f+\pi_{Y}Bu,
\]
respectively. We may rewrite this as 
\begin{align}
 & \left(\begin{array}{cc}
\partial_{0}\pi_{V}K(\partial_{0}^{-1})\pi_{V}^{\ast} & \pi_{V}M_{1,12}\\
M_{1,21}\pi_{V}^{\ast} & M_{1,22}
\end{array}\right)\left(\begin{array}{c}
w\\
y
\end{array}\right)\nonumber \\
 & =\left(\begin{array}{c}
\pi_{V}f+\pi_{V}Bu-\partial_{0}\pi_{V}K(\partial_{0}^{-1})(\pi_{H_{0}(|G|)}^{\ast}x+\pi_{H_{0}(|\interior D|)}^{\ast}\zeta)-Cx\\
\pi_{Y}f+\pi_{Y}Bu-M_{1,20}x-M_{1,21}\pi_{H_{0}(|\interior D|)}^{\ast}\zeta
\end{array}\right)\label{eq:control_eq}
\end{align}
or equivalently as 
\begin{align}
 & \left(\begin{array}{cc}
\partial_{0}\pi_{V}K(\partial_{0}^{-1})\pi_{V}^{\ast} & -\pi_{V}B\\
M_{1,21}\pi_{V}^{\ast} & -\pi_{Y}B
\end{array}\right)\left(\begin{array}{c}
w\\
u
\end{array}\right)\nonumber \\
 & =\left(\begin{array}{c}
\pi_{V}f-\pi_{V}M_{1,12}y-\partial_{0}\pi_{V}K(\partial_{0}^{-1})(\pi_{H_{0}(|G|)}^{\ast}x+\pi_{H_{0}(|\interior D|)}^{\ast}\zeta)-Cx\\
\pi_{Y}f-M_{1,22}y-M_{1,20}x-M_{1,21}\pi_{H_{0}(|\interior D|)}^{\ast}\zeta
\end{array}\right)\label{eq:observation_eq}
\end{align}
If the material law $M$ satisfies the solvability condition (\ref{eq:pos_def}),
then the operator on the left hand side of (\ref{eq:control_eq})
is boundedly invertible and thus, we can express $w$ in terms of
$x,\zeta,f$ and $u.$ Plugging this representation into (\ref{eq:bd_cd})
we obtain the boundary control equation. Analogously, by assuming
that the operator on the left hand side of Equation (\ref{eq:observation_eq})
is boundedly invertible, we can express $w$ in terms of $x,\zeta,f$
and $y$ and hence, Equation (\ref{eq:bd_cd}) yields the boundary
observation equation.%

\end{rem}

\section{A Boundary Control Problem in Visco-elasticity\label{sec:Applications}}

In this section we apply our results to a boundary control problem
for the equations of visco-elasticity. For this purpose we introduce
the required differential operators. Throughout let $\Omega\subseteq\mathbb{R}^{n}$
be an open subset of $\mathbb{R}^{n}$, $n\in\mathbb{N}$, $n\geq1$.
\begin{defn}
We denote by $L_{2}(\Omega)^{n\times n}$ the Hilbert space of $n\times n$-matrices
with entries in $L_{2}(\Omega)$ endowed with the inner product 
\[
\langle\Phi|\Psi\rangle_{L_{2}(\Omega)^{n\times n}}\coloneqq\intop_{\Omega}\trace\left(\Phi(x)^{\ast}\Psi(x)\right)\mbox{ d}x\quad\left(\Phi,\Psi\in L_{2}(\Omega)^{n\times n}\right).
\]
Moreover let $H_{\mathrm{sym}}(\Omega)\subseteq L_{2}(\Omega)^{n\times n}$
denote the closed subspace of symmetric $n\times n$ matrices. We
define the operator $\interior\Grad$ as the closure of 
\begin{align*}
\Grad|_{C_{c}^{\infty}(\Omega)^{n}}:C_{c}^{\infty}(\Omega)^{n}\subseteq L_{2}(\Omega)^{n} & \to H_{\mathrm{sym}}(\Omega)\\
(\phi_{i})_{i\in\{1,\ldots,n\}} & \mapsto\left(\frac{1}{2}\left(\partial_{j}\phi_{i}+\partial_{i}\phi_{j}\right)\right)_{i,j\in\{1,\ldots,n\}}
\end{align*}
and $\interior\Dive$ as the closure of 
\begin{align*}
\Dive|_{C_{c,\mathrm{sym}}^{\infty}(\Omega)^{n\times n}}:C_{c,\mathrm{sym}}^{\infty}(\Omega)^{n\times n}\subseteq H_{\mathrm{sym}}(\Omega) & \to L_{2}(\Omega)^{n}\\
(\Phi_{ij})_{i,j\in\{1,\ldots,n\}} & \mapsto\left(\sum_{j=1}^{n}\partial_{j}\Phi_{ij}\right)_{i\in\{1,\ldots,n\}},
\end{align*}
where we denote by $C_{c,\mathrm{sym}}^{\infty}(\Omega)^{n\times n}$
the space of symmetric $n\times n$-matrices with entries in $C_{c}^{\infty}(\Omega).$
\\
Furthermore we extend the meaning of $\interior\grad$ by defining
it as the closure of 
\begin{align*}
\grad|_{C_{c}^{\infty}(\Omega)^{n}}:C_{c}^{\infty}(\Omega)^{n}\subseteq L_{2}(\Omega)^{n} & \to L_{2}(\Omega)^{n\times n}\\
(\psi_{i})_{i\in\{1,\ldots,n\}} & \mapsto\left(\partial_{j}\psi_{i}\right)_{i,j\in\{1,\ldots,n\}}
\end{align*}
and similarly $\interior\dive$ as the closure of 
\begin{align*}
\dive|_{C_{c}^{\infty}(\Omega)^{n\times n}}:C_{c}^{\infty}(\Omega)^{n\times n}\subseteq L_{2}(\Omega)^{n\times n} & \to L_{2}(\Omega)^{n}\\
(\Psi_{ij})_{i,j\in\{1,\ldots,n\}} & \mapsto\left(\sum_{j=1}^{n}\partial_{j}\Psi_{ij}\right)_{i\in\{1,\ldots,n\}}
\end{align*}
leaving it to the context to determine if the scalar or the matrix
version of these operations are meant.
\end{defn}
An easy computation shows that $\interior\Grad$ and $\interior\Dive$
are formally skew-adjoint, likewise the extended operations $\interior\grad$
and $\interior\dive$ are formally skew-adjoint. Following the notation
introduced in Section 4 we define $\Grad\coloneqq-(\interior\Dive)^{\ast}$,
$\Dive\coloneqq-(\interior\Grad)^{\ast},$ $\grad\coloneqq-(\interior\dive)^{\ast}$
and $\dive\coloneqq-(\interior\grad)^{\ast}.$ The equations of visco-elasticity
are given by 
\begin{align}
\partial_{0}^{2}\rho x(t)-\Dive T(t) & =f(t),\label{eq:dispalcement}\\
T(t) & =M\Grad x(t)-\intop_{-\infty}^{t}g(t-s)\Grad x(s)\mbox{ d}s.\label{eq:stress}
\end{align}

Here $x\in H_{\nu,0}(\mathbb{R};L_{2}(\Omega)^{n})$ and $T\in H_{\nu,0}\left(\mathbb{R};H_{\mathrm{sym}}(\Omega)\right)$
are the unknowns, denoting the displacement field and the stress tensor,
respectively. The density function $\rho\in L_{\infty}(\Omega)$ is
assumed to be real-valued and uniformly strictly positive, i.e. $\rho\geq c_{1}>0.$
The tensor $M\in L\left(H_{\mathrm{sym}}(\Omega)\right)$, linking
the stress and the strain tensor, is assumed to be selfadjoint and
satisfies $M\geq c_{2}>0$. The function $g:\mathbb{R}_{\geq0}\to\mathbb{C}$
is assumed to be absolutely continuous%
\footnote{The equations can also be studied in a more general setting, for instance
$g$ can attain values in $L(H_{\mathrm{sym}}(\Omega))$ (cf. \cite{Trostorff2012_integro}).%
}, i.e. $g(t)=\intop_{0}^{t}h(s)\mbox{ d}s+g_{0}$ for some $h\in L_{1}(\mathbb{R}_{\geq0})$
and $g_{0}\in\mathbb{C}$. An easy computation shows that the convolution
operator $g\ast:H_{\nu,0}(\mathbb{R};H_{\mathrm{sym}}(\Omega))\to H_{\nu,0}(\mathbb{R};H_{\mathrm{sym}}(\Omega))$
is continuous for each $\nu>0$ with $\|g\ast\|_{L(H_{\nu,0}(\mathbb{R};H_{\mathrm{sym}}(\Omega))}\leq\frac{1}{\nu}\left(|h|_{L_{1}(\mathbb{R}_{\geq0})}+|g_{0}|\right).$
Thus, for $\nu>0$ large enough, the operator $(1-M^{-1}g\ast)$ is
invertible, and hence we may write (\ref{eq:stress}) as 
\begin{equation}
(M-g\ast)^{-1}T=(1-M^{-1}g\ast)^{-1}M^{-1}T=\Grad x.\label{eq:stress_new}
\end{equation}

The boundary control and observation equations are given by 
\begin{align}
TN & =\partial_{0}x+\sqrt{2}u,\nonumber \\
TN & =\sqrt{2}y-\partial_{0}x,\label{eq:control}
\end{align}
on $\partial\Omega$, where we denote by $N$ the outer unit normal
vector field.
\begin{rem}
Since we have to compare Neumann-type traces and Dirichlet-type traces
we have to determine a suitable control and observation space. For
doing so, let us assume for the moment that $\partial\Omega$ is smooth.
We consider the space $L_{2}(\partial\Omega)^{n}.$ We assume that
the outer unit normal vector field $N$ can be extended to $\Omega$
such that $N\in L_{\infty}(\Omega;\mathbb{R})^{n}$ and $\dive N\in L_{\infty}(\Omega).$
For $f,g\in\mathrm{BD}(\grad)$ we formally compute using the divergence
theorem%
\footnote{Note that due to the assumptions on the vector field $N$, the matrix-valued
function $(f_{i}N_{k})_{i,k\in\{1,\ldots,n\}}$ lies in $D(\dive)$
for each $f\in D(\grad).$%
} 
\begin{align*}
\intop_{\partial\Omega}f\cdot g\mbox{ d}S & =\intop_{\partial\Omega}\left(f\cdot g\right)\left(N\cdot N\right)\mbox{ d}S\\
 & =\frac{1}{2}\intop_{\partial\Omega}\left(\left(N_{k}f_{i}\right){}_{k,i}g^{*}\right)\cdot N\mbox{ d}S+\frac{1}{2}\intop_{\partial\Omega}\left(\left(N_{k}g_{i}^{*}\right)_{k,i}f\right)\cdot N\mbox{ d}S\\
 & =\frac{1}{2}\intop_{\Omega}\dive\left(\left(N_{k}f_{i}^{*}\right){}_{k,i}g\right)+\frac{1}{2}\intop_{\Omega}\dive\left(\left(N_{k}g_{i}\right)_{k,i}f^{*}\right)\\
 & =\frac{1}{2}\left(\left\langle \left.\dive\left(f_{i}N_{k}\right)_{i,k}\right|g\right\rangle _{L_{2}(\Omega)^{n}}+\left\langle \left.\left(f_{i}N_{k}\right){}_{i,k}\right|\grad g\right\rangle _{L_{2}(\Omega)^{n\times n}}\right)\\
 & \quad+\frac{1}{2}\left(\left\langle f\left|\dive\left(g_{i}N_{k}\right)_{i,k}\right.\right\rangle _{L_{2}(\Omega)^{n}}+\left\langle \grad f\left|\left(g_{i}N_{k}\right)_{i,k}\right.\right\rangle _{L_{2}(\Omega)^{n\times n}}\right)\\
 & =\frac{1}{2}\left(\left\langle \left.\pi_{\mathrm{BD}(\dive)}\left(f_{i}N_{k}\right)_{i,k}\right|\stackrel{\bullet}{\grad}g\right\rangle _{\mathrm{BD}(\dive)}+\left\langle \stackrel{\bullet}{\grad}f\left|\pi_{\mathrm{BD}(\dive)}\left(g_{i}N_{k}\right)_{i,k}\right.\right\rangle _{\mathrm{BD}(\dive)}\right).
\end{align*}
This leads to the following choice for the control space $U.$\end{rem}
\begin{defn}
Let $N\in L_{\infty}(\Omega;\mathbb{R})^{n}$ be such that $\dive N\in L_{\infty}.$
We define the bounded linear operator $\nu:BD(\grad)\to\mathrm{BD}(\dive)$
by $\nu f\coloneqq\pi_{\mathrm{BD}(\dive)}\left(f_{i}N_{k}\right)_{i,k\in\{1,\ldots,n\}}$.
We assume that the operator $\stackrel{\bullet}{\dive}\nu+\nu^{\ast}\stackrel{\bullet}{\grad}$
is positive, i.e. for every $f\in\mathrm{BD}(\grad)\setminus\{0\}$
we have 
\[
\left\langle \left.\left(\stackrel{\bullet}{\dive}\nu+\nu^{\ast}\stackrel{\bullet}{\grad}\right)f\,\right|f\right\rangle _{\mathrm{BD}(\grad)}>0.
\]
We define the Hilbert space $U$ as the completion of $\mathrm{BD}(\grad)$
with respect to the inner product 
\begin{align*}
\langle\cdot|\cdot\rangle_{U}:\mathrm{BD}(\grad)\times\mathrm{BD}(\grad) & \to\mathbb{C}\\
(f,g) & \mapsto\frac{1}{2}\left(\left\langle \nu f\left|\stackrel{\bullet}{\grad}g\right.\right\rangle _{\mathrm{BD}(\dive)}+\left\langle \left.\stackrel{\bullet}{\grad}f\right|\nu g\right\rangle _{\mathrm{BD}(\dive)}\right).
\end{align*}
We denote the embedding $\mathrm{BD}(\grad)\hookrightarrow U$ by
$\iota.$
\end{defn}
In the following we require that Korn's inequality holds, i.e., $H_{1}(|\Grad|+\i)\stackrel{\kappa}{\hookrightarrow}H_{1}(|\grad|+\i)$
(for sufficient criteria see \cite{Ciarlet2005} and the references
therein). We consider the bounded operator $j:\mathrm{BD}(\Grad)\to U$
given by $j=\iota\circ\pi_{\mathrm{BD}(\grad)}\circ\kappa\circ\pi_{\mathrm{BD}(\Grad)}^{\ast}$
and compute 
\begin{align*}
\langle jf|g\rangle_{U} & =\langle\pi_{\mathrm{BD}(\grad)}\kappa\pi_{\mathrm{BD}(\Grad)}^{*}f|g\rangle_{U}\\
 & =\frac{1}{2}\left(\left\langle \nu\pi_{\mathrm{BD}(\grad)}\kappa\pi_{\mathrm{BD}(\Grad)}^{*}f\left|\stackrel{\bullet}{\grad}g\right.\right\rangle +\left\langle \left.\stackrel{\bullet}{\grad}\pi_{\mathrm{BD}(\grad)}\kappa\pi_{\mathrm{BD}(\Grad)}^{*}f\right|\nu g\right\rangle _{\mathrm{BD}(\dive)}\right)\\
 & =\frac{1}{2}\left(\left\langle f\left|\pi_{\mathrm{BD}(\Grad)}\kappa^{\ast}\pi_{\mathrm{BD}(\grad)}^{\ast}\nu^{\ast}\stackrel{\bullet}{\grad}g\right.\right\rangle _{\mathrm{BD}(\Grad)}\right.\\
 & \quad\left.+\left\langle f\left|\pi_{\mathrm{BD}(\Grad)}\kappa^{\ast}\pi_{\mathrm{BD}(\grad)}^{\ast}\stackrel{\bullet}{\dive}\nu g\right.\right\rangle _{\mathrm{BD}(\Grad)}\right)\\
 & =\left\langle f\left|\frac{1}{2}\pi_{\mathrm{BD}(\Grad)}\kappa^{\ast}\pi_{\mathrm{BD}(\grad)}^{\ast}\left(\nu^{\ast}\stackrel{\bullet}{\grad}+\stackrel{\bullet}{\dive}\nu\right)g\right.\right\rangle _{\mathrm{BD}(\Grad)}
\end{align*}

for each $f\in\mathrm{BD}(\Grad),\, g\in\mathrm{BD}(\grad).$ This
gives
\[
j^{*}=\frac{1}{2}\pi_{\mathrm{BD}(\Grad)}\kappa^{\ast}\pi_{\mathrm{BD}(\grad)}^{\ast}\left(\nu^{\ast}\stackrel{\bullet}{\grad}+\stackrel{\bullet}{\dive}\nu\right)
\]
and, consequently, 
\begin{align}
 & \gamma_{\Dive}\pi_{\mathrm{BD}(\Dive)}^{\ast}\stackrel{\bullet}{\Grad}j^{\ast}g\label{eq:j}\\
\notag & =\frac{1}{2}\gamma_{\Dive}\pi_{\mathrm{BD}(\Dive)}^{\ast}\stackrel{\bullet}{\Grad}\pi_{\mathrm{BD}(\Grad)}\kappa^{\ast}\pi_{\mathrm{BD}(\grad)}^{\ast}\left(\nu^{\ast}\stackrel{\bullet}{\grad}+\stackrel{\bullet}{\dive}\nu\right)g.
\end{align}

\begin{rem}
\label{rem:abstract_bd_cd}We give a possible interpretation of the
latter equality. For this purpose we compute formally using the divergence
theorem
\begin{align*}
 & \intop_{\partial\Omega}\left(\Grad j^{\ast}g\right)N\cdot f\mbox{ d}S\\
 & =\langle\gamma_{\Dive}\pi_{\mathrm{BD}(\Dive)}^{\ast}\stackrel{\bullet}{\Grad}j^{\ast}g|\pi_{\mathrm{BD}(\Grad)}^{\ast}f\rangle\\
 & =\frac{1}{2}\langle\gamma_{\Dive}\pi_{\mathrm{BD}(\Dive)}^{\ast}\stackrel{\bullet}{\Grad}\pi_{\mathrm{BD}(\Grad)}\kappa^{\ast}\pi_{\mathrm{BD}(\grad)}^{\ast}\nu^{\ast}\stackrel{\bullet}{\grad}g|\pi_{\mathrm{BD}(\Grad)}^{\ast}f\rangle\\
 & \quad+\frac{1}{2}\langle\gamma_{\Dive}\pi_{\mathrm{BD}(\Dive)}^{\ast}\stackrel{\bullet}{\Grad}\pi_{\mathrm{BD}(\Grad)}\kappa^{\ast}\pi_{\mathrm{BD}(\grad)}^{\ast}\stackrel{\bullet}{\dive}\nu g|\pi_{\mathrm{BD}(\Grad)}^{\ast}f\rangle\\
 & =\frac{1}{2}\langle\pi_{\mathrm{BD}(\Grad)}\kappa^{\ast}\pi_{\mathrm{BD}(\grad)}^{\ast}\nu^{\ast}\stackrel{\bullet}{\grad}g|f\rangle+\frac{1}{2}\langle\stackrel{\bullet}{\Grad}\pi_{\mathrm{BD}(\Grad)}\kappa^{\ast}\pi_{\mathrm{BD}(\grad)}^{\ast}\nu^{\ast}\stackrel{\bullet}{\grad}g|\stackrel{\bullet}{\Grad}f\rangle\\
 & \quad+\frac{1}{2}\langle\pi_{\mathrm{BD}(\Grad)}\kappa^{\ast}\pi_{\mathrm{BD}(\grad)}^{\ast}\stackrel{\bullet}{\dive}\nu g|f\rangle+\frac{1}{2}\langle\stackrel{\bullet}{\Grad}\pi_{\mathrm{BD}(\Grad)}\kappa^{\ast}\pi_{\mathrm{BD}(\grad)}^{\ast}\stackrel{\bullet}{\dive}\nu g|\stackrel{\bullet}{\Grad}f\rangle\\
 & =\frac{1}{2}\langle\pi_{\mathrm{BD}(\Grad)}\kappa^{\ast}\pi_{\mathrm{BD}(\grad)}^{\ast}\nu^{\ast}\stackrel{\bullet}{\grad}g|f\rangle_{\mathrm{BD}(\Grad)}+\frac{1}{2}\langle\pi_{\mathrm{BD}(\Grad)}\kappa^{\ast}\pi_{\mathrm{BD}(\grad)}^{\ast}\stackrel{\bullet}{\dive}\nu g|f\rangle_{\mathrm{BD}(\Grad)}\\
 & =\frac{1}{2}\langle\stackrel{\bullet}{\grad}g|\nu\pi_{\mathrm{BD}(\grad)}\kappa\pi_{\mathrm{BD}(\Grad)}^{\ast}f\rangle_{\mathrm{BD}(\dive)}+\frac{1}{2}\langle\nu g|\stackrel{\bullet}{\grad}\pi_{\mathrm{BD}(\grad)}\kappa\pi_{\mathrm{BD}(\Grad)}^{\ast}f\rangle_{\mathrm{BD}(\dive)}\\
 & =\langle g|\pi_{\mathrm{BD}(\grad)}\kappa\pi_{\mathrm{BD}(\Grad)}^{\ast}f\rangle_{U}\\
 & =\intop_{\partial\Omega}g\cdot\pi_{\mathrm{BD}(\grad)}\kappa\pi_{\mathrm{BD}(\Grad)}^{\ast}f\mbox{ d}S\\
 & =\intop_{\partial\Omega}g\cdot\pi_{\mathrm{BD}(\grad)}\kappa\pi_{\mathrm{BD}(\Grad)}^{\ast}f\mbox{ d}S+\intop_{\partial\Omega}g\cdot(1-\pi_{\mathrm{BD}(\grad)})\kappa\pi_{\mathrm{BD}(\Grad)}^{\ast}f\mbox{ d}S\\
 & =\intop_{\partial\Omega}g\cdot\kappa\pi_{\mathrm{BD}(\Grad)}^{\ast}f\mbox{ d}S\\
 & =\intop_{\partial\Omega}g\cdot f\mbox{ d}S
\end{align*}
for each $f\in\mathrm{BD}(\Grad),g\in\mathrm{BD}(\grad).$ Hence,
equality (\ref{eq:j}) can be seen as a generalization of
\begin{equation}
\left(\Grad j^{\ast}g\right)N=g\mbox{ on }\partial\Omega\label{eq:comparison_boundary_data}
\end{equation}
to the case of non-smooth boundaries.
\end{rem}
We now want to transform the equations (\ref{eq:dispalcement}), (\ref{eq:stress_new})
and (\ref{eq:control}) into a system of the form treated in the previous
subsections. In the terminology of Subsection 4 the operator $\Grad$
should play the role of $G$ and $\Dive$ the role of $D.$ Since
we have assumed in Theorem \ref{thm:adjoint_F} that $G$ is boundedly
invertible, we require that $\Grad[L_{2}(\Omega)^{n}]$ is closed
in $H_{\mathrm{sym}}(\Omega)$.%
\footnote{The closedness of the range $\Grad[L_{2}(\Omega)^{n}]$ holds, for
instance if $H_{1}(|\Grad|+\i)$ is compactly embedded in $L_{2}(\Omega)^{n}$
and we refer to \cite{Weck1994} for sufficient conditions on $\Omega$
yielding this compact embedding. Note that this compact embedding
yields then a Poincare-type estimate, which in turn yields the closedness
of the range. Thus, the minimal assumption is the validity of a Poincare-type
estimate.%
} Then the projection theorem yields the following orthogonal decompositions%
\footnote{Note that the closedness of the range of $\Grad$ also yields the
closedness of range of $\stackrel{\circ}{\Dive}$. Since we do not
want to give the details of the proof here, we use the closure bar
for convenience.%
} 
\begin{align*}
L_{2}(\Omega)^{n} & =[\{0\}]\Grad\oplus\overline{\stackrel{\circ}{\Dive}[H_{\mathrm{sym}}(\Omega)]},\\
H_{\mathrm{sym}}(\Omega) & =[\{0\}]\stackrel{\circ}{\Dive}\oplus\Grad[L_{2}(\Omega)^{n}].
\end{align*}

We define the orthogonal projections $\pi_{\Dive}:L_{2}(\Omega)^{n}\to\overline{\stackrel{\circ}{\Dive}[H_{\mathrm{sym}}(\Omega)]}$
and $\pi_{\Grad}:H_{\mathrm{sym}}(\Omega)\to\Grad[L_{2}(\Omega)^{n}].$
Note that due to the closed graph theorem the operator $\widetilde{\Grad}\coloneqq\pi_{\Grad}\Grad\pi_{\Dive}^{\ast}$
is boundedly invertible and so is $\left(\widetilde{\Grad}\right)^{\ast}=-\pi_{\Dive}\stackrel{\circ}{\Dive}\pi_{\Grad}^{\ast}$.
Furthermore let us denote by $\iota_{\Grad}$ the embedding $H_{1}\left(\left|\widetilde{\Grad}\right|+\i\right)\hookrightarrow H_{1}\left(\left|\Grad\right|+\i\right).$
We consider the following evolutionary problem 
\begin{align}
\left(\partial_{0}\left(\begin{array}{ccc}
\pi_{\Dive}\rho\pi_{\Dive}^{\ast} & \left(\begin{array}{cc}
0 & 0\end{array}\right) & 0\\
\left(\begin{array}{c}
0\\
0
\end{array}\right) & \left(\begin{array}{cc}
\pi_{\Grad}(M-g\ast)^{-1}\pi_{\Grad}^{\ast} & 0\\
0 & \partial_{0}^{-1}
\end{array}\right) & \left(\begin{array}{c}
0\\
0
\end{array}\right)\\
0 & \left(\begin{array}{cc}
0 & 0\end{array}\right) & 0
\end{array}\right)+\left(\begin{array}{ccc}
0 & \left(\begin{array}{cc}
0 & 0\end{array}\right) & 0\\
\left(\begin{array}{c}
0\\
0
\end{array}\right) & \left(\begin{array}{cc}
0 & 0\\
0 & 0
\end{array}\right) & \left(\begin{array}{c}
0\\
0
\end{array}\right)\\
0 & \left(\begin{array}{cc}
0 & \sqrt{2}\end{array}\right) & 1
\end{array}\right)\right.\nonumber \\
\left.+\left(\begin{array}{ccc}
0 & \left(\begin{array}{cc}
-\left(\widetilde{\Grad}\right)^{\ast} & -C^{\diamond}\end{array}\right) & 0\\
\left(\begin{array}{c}
-\widetilde{\Grad}\\
C
\end{array}\right) & \left(\begin{array}{cc}
0 & 0\\
0 & 0
\end{array}\right) & \left(\begin{array}{c}
0\\
0
\end{array}\right)\\
0 & \left(\begin{array}{cc}
0 & 0\end{array}\right) & 0
\end{array}\right)\right)\left(\begin{array}{c}
v\\
\left(\begin{array}{c}
T\\
w
\end{array}\right)\\
y
\end{array}\right)=\left(\begin{array}{c}
f\\
\left(\begin{array}{c}
0\\
0
\end{array}\right)\\
0
\end{array}\right)+\left(\begin{array}{c}
0\\
\left(\begin{array}{c}
0\\
-\sqrt{2}
\end{array}\right)\\
-1
\end{array}\right)u,\label{eq:control_elasto}
\end{align}

where $C:H_{1}\left(\left|\widetilde{\Grad}\right|+\i\right)\to U$
is given by $Cx\coloneqq j\pi_{\mathrm{BD}(\Grad)}\iota_{\Grad}x.$
The material law $K$ is given by
\[
K(z)=\left(\begin{array}{ccc}
\pi_{\Dive}\rho\pi_{\Dive}^{\ast} & 0 & 0\\
0 & \pi_{\Grad}\left(M-\sqrt{2\pi}\hat{g}(-\i z^{-1})\right)^{-1}\pi_{\Grad}^{\ast} & 0\\
0 & 0 & z
\end{array}\right)
\]

and satisfies the solvability condition (\ref{eq:pos_def}). Indeed,
using the representation $z^{-1}=\i t+\nu$ for some $\nu>\frac{1}{2r},t\in\mathbb{R}$
if $z\in B_{\mathbb{C}}(r,r)$ we estimate
\[
\Re z^{-1}\pi_{\Dive}\rho\pi_{\Dive}^{\ast}\geq\nu c_{1}
\]

and 
\begin{align*}
 & \Re z^{-1}\pi_{\Grad}\left(M-\sqrt{2\pi}\hat{g}(-\i z^{-1})\right)^{-1}\pi_{\Grad}^{\ast}\\
 & =\Re z^{-1}\pi_{\Grad}\left(\sum_{k=0}^{\infty}\left(2\pi\right)^{\frac{k}{2}}M^{-k}\hat{g}(-\i z^{-1})^{k}\right)M^{-1}\pi_{\Grad}^{\ast}\\
 & =\nu\pi_{\Grad}M^{-1}\pi_{\Grad}^{\ast}\\
 & \quad+\Re\pi_{\Grad}z^{-1}\left(\sqrt{2\pi}M^{-1}\hat{g}(-\i z^{-1})\right)\left(\sum_{k=0}^{\infty}\left(2\pi\right)^{\frac{k}{2}}M^{-k}\hat{g}(-\i z^{-1})^{k}\right)M^{-1}\pi_{\Grad}^{\ast}\\
 & =\nu\pi_{\Grad}M^{-1}\pi_{\Grad}^{\ast}\\
 & \quad+\Re\pi_{\Grad}M^{-1}\left(\sqrt{2\pi}\hat{h}(-\i z^{-1})+g_{0}\right)\left(\sum_{k=0}^{\infty}\left(2\pi\right)^{\frac{k}{2}}M^{-k}\hat{g}(-\i z^{-1})^{k}\right)M^{-1}\pi_{\Grad}^{\ast}\\
 & \geq\nu c_{2}-\frac{\|M^{-1}\|^{2}\left(|h|_{L_{1}(\mathbb{R}_{\geq0})}+|g_{0}|\right)}{1-\nu^{-1}\|M^{-1}\|\left(|h|_{L_{1}(\mathbb{R}_{\geq0})}+|g_{0}|\right)},
\end{align*}

where we have used $\hat{g}(-\i z^{-1})=z\hat{h}(-\i z^{-1})+\frac{z}{\sqrt{2\pi}}g_{0}.$
Summarizing this gives $\Re z^{-1}K(z)\geq1.$ The operator $J$ is
given by $J=\left(\begin{array}{c}
0\\
\left(\begin{array}{c}
0\\
\frac{1}{\sqrt{2}}
\end{array}\right)
\end{array}\right)$ and thus, $\|J\|=\frac{1}{\sqrt{2}}.$ Since $\Re M_{1,22}=1$ Theorem
\ref{thm:well_posedness_control} applies and thus, the control system
given by (\ref{eq:control_elasto}) is well-posed. Next, we compute
$C^{\diamond}$. For that purpose let $x\in H_{1}\left(\left|\widetilde{\Grad}\right|+\i\right)$
and $w\in U.$ Then 
\begin{align*}
\langle C^{\diamond}w|x\rangle & =\langle w|Cx\rangle_{U}\\
 & =\langle w|j\pi_{\mathrm{BD}(\Grad)}\iota_{\Grad}x\rangle_{U}\\
 & =\langle\pi_{\mathrm{BD}(\Grad)}^{\ast}j^{\ast}w|\iota_{\Grad}x\rangle_{H_{1}(|\Grad|+\i)}\\
 & =\langle\pi_{\mathrm{BD}(\Grad)}^{\ast}j^{\ast}w|\iota_{\Grad}x\rangle+\langle\Grad\pi_{\mathrm{BD}(\Grad)}^{\ast}j^{\ast}w|\Grad\iota_{\Grad}x\rangle\\
 & =\langle\pi_{\mathrm{BD}(\Grad)}^{\ast}j^{\ast}w|x\rangle+\left\langle \pi_{\Grad}\pi_{\mathrm{BD}(\Dive)}^{\ast}\stackrel{\bullet}{\Grad}j^{\ast}w\left|\widetilde{\Grad}x\right.\right\rangle \\
 & =\langle\pi_{\mathrm{BD}(\Grad)}^{\ast}j^{\ast}w|x\rangle+\left\langle \left.\left(\widetilde{\Grad}\right)^{*}\pi_{\Grad}\pi_{\mathrm{BD}(\Dive)}^{\ast}\stackrel{\bullet}{\Grad}j^{\ast}w\right|x\right\rangle 
\end{align*}
Summarizing, we get that $C^{\diamond}=\pi_{\mathrm{BD}(\Grad)}^{\ast}j^{\ast}+\left(\widetilde{\Grad}\right)^{*}\pi_{\Grad}\pi_{\mathrm{BD}(\Dive)}^{\ast}\stackrel{\bullet}{\Grad}j^{\ast}.$
According to the definition of the domain of $\left(\begin{array}{cc}
-\left(\widetilde{\Grad}\right)^{\ast} & -C^{\diamond}\end{array}\right)$ the implicit boundary condition for the system reads as 
\[
\left(\widetilde{\Grad}\right)^{*}T+C^{\diamond}w\in H_{0}\left(\left|\widetilde{\Grad}\right|+\i\right).
\]
Hence, 
\[
T+\left(\left(\widetilde{\Grad}\right)^{*}\right)^{-1}C^{\diamond}w\in H_{1}\left(\left|\left(\widetilde{\Grad}\right)^{*}\right|+\i\right)\subseteq H_{1}\left(\left|\interior\Dive\right|+\i\right)\subseteq H_{1}\left(\left|\Dive\right|+\i\right).
\]
From 
\[
\pi_{\Grad}\pi_{\mathrm{BD}(\Dive)}^{\ast}\stackrel{\bullet}{\Grad}j^{\ast}w=\pi_{\Grad}\Grad\pi_{\mathrm{BD}(\Grad)}^{\ast}j^{\ast}w\in H_{1}(|\Dive|+\i)
\]
it thus follows that $T\in H_{1}\left(|\Dive|+\i\right)$ and 
\begin{align*}
\gamma_{\Dive}T & =\gamma_{\Dive}\left(-\left(\left(\widetilde{\Grad}\right)^{*}\right)^{-1}C^{\diamond}w\right)\\
 & =\gamma_{\Dive}\left(-\left(\left(\widetilde{\Grad}\right)^{*}\right)^{-1}\pi_{\Grad}\pi_{\mathrm{BD}(\Grad)}^{\ast}j^{\ast}w-\pi_{\Grad}\pi_{\mathrm{BD}(\Dive)}^{\ast}\stackrel{\bullet}{\Grad}j^{\ast}w\right)\\
 & =-\gamma_{\Dive}\pi_{\Grad}\pi_{\mathrm{BD}(\Dive)}^{\ast}\stackrel{\bullet}{\Grad}j^{\ast}w\\
 & =-\gamma_{\Dive}\pi_{\mathrm{BD}(\Dive)}^{\ast}\stackrel{\bullet}{\Grad}j^{\ast}w,
\end{align*}
where we used that $\gamma_{\Dive}$ vanishes on the domain of $\stackrel{\circ}{\Dive}$,
which is a superset of the domain of $\left(\widetilde{\Grad}\right)^{*}$.
Using (\ref{eq:control_eq}) we get 
\[
\left(\begin{array}{cc}
1 & 0\\
\sqrt{2} & 1
\end{array}\right)\left(\begin{array}{c}
w\\
y
\end{array}\right)=\left(\begin{array}{c}
-\sqrt{2}u-Cv\\
-u
\end{array}\right)
\]

and hence 
\[
w=u-Cv=-\sqrt{2}u-j\pi_{\mathrm{BD}(\Grad)}\iota_{\Grad}v.
\]

Analogously one obtains, using (\ref{eq:observation_eq}), that
\[
\left(\begin{array}{cc}
1 & \sqrt{2}\\
\sqrt{2} & 1
\end{array}\right)\left(\begin{array}{c}
w\\
u
\end{array}\right)=\left(\begin{array}{c}
-Cv\\
-y
\end{array}\right)
\]

and thus, 
\[
w=Cv-\sqrt{2}y=j\pi_{\mathrm{BD}(\Grad)}\iota_{\Grad}v-\sqrt{2}y.
\]

Following the reasoning of Remark \ref{rem:abstract_bd_cd}, we may
interpret the resulting boundary control and observation equations
as
\begin{align*}
TN & =v-u=\partial_{0}x+\sqrt{2}u,\\
TN & =\sqrt{2}y-v=\sqrt{2}y-\partial_{0}x
\end{align*}
on $\partial\Omega.$

\end{document}